\documentclass[reqno, 11pt]{amsart}

\usepackage[parfill]{parskip}    
\usepackage{amsrefs, amsmath}
\usepackage{amsfonts, mathrsfs}
\usepackage{amssymb}
\usepackage{amscd, mathtools}
\usepackage{cleveref, multirow} 
\usepackage{fullpage, color}
\usepackage[labelformat=empty]{subfig}
\usepackage{booktabs}
\usepackage{xypic}
\usepackage{enumerate, longtable}
\usepackage{makecell}

\numberwithin{equation}{section}

\BibSpec{article}{%
+{}{\sc \PrintAuthors} {author}
+{,}{ \textrm} {title}
+{,}{ \textit} {journal}
+{,}{ } {volume}
+{}{ \IfEmptyBibField{journal}{}{\PrintDate{date}}} {transition}
+{,}{ } {pages}
+{,}{ } {status}
+{.}{} {transition}
+{}{ Preprint available at \tt} {eprint}
+{.}{} {transition}
}

\BibSpec{book}{%
+{}{\sc \PrintAuthors} {author}
+{,}{ \textit} {title}
+{.}{ } {series}
+{,}{ } {volume}
+{.}{ } {publisher}
+{,}{ } {place}
+{,}{ } {date}
+{.}{} {transition}
}

\BibSpec{collection.article}{
+{} {\sc \PrintAuthors} {author}
+{,} { \textrm } {title}
+{,} { in \it \PrintConference} {conference}
+{,} { pp.~} {pages}
+{.} {\PrintBook} {book}
+{,} { \PrintDateB} {date}
+{.} {} {transition}
}

\BibSpec{innerbook}{
+{.} { \emph} {title}
+{.} { } {part}
+{:} { \emph} {subtitle}
+{.} { } {series}
+{,} { } {volume}
+{.} { Edited by \PrintNameList} {editor}
+{.} { Translated by \PrintNameList}{translator}
+{.} { \PrintContributions} {contribution}
+{.} { } {publisher}
+{.} { } {organization}
+{,} { } {address}
+{,} { \PrintEdition} {edition}
+{,} { \PrintDateB} {date}
+{.} { } {note}
}

\allowdisplaybreaks[3]

\crefname{equation}{Eq.}{Eqs.}
\crefname{eqnarray}{Eq.}{Eqs.}
\crefname{algo}{Algorithm}{Algorithms}
\crefname{conj}{Conjecture}{Conjectures}
\crefname{defn}{Definition}{Definitions}
\crefname{lem}{Lemma}{Lemmas}
\crefname{thm}{Theorem}{Theorems}
\crefname{claim}{Claim}{Claims}
\crefname{rmk}{Remark}{Remarks}
\crefname{prop}{Proposition}{Propositions}
\crefname{section}{Section}{Sections}
\crefname{appendix}{Appendix}{Appendices}
\crefname{cor}{Corollary}{Corollaries}
\crefname{figure}{Figure}{Figures}
\crefname{table}{Table}{Tables}
\crefname{example}{Example}{Examples}
\crefname{prob}{Problem}{Problems}
\crefname{assm}{Assumption}{Assumptions}

\newcommand{\de}{{\partial}}

\newcommand{\re}{\mathrm{e}}

\newcommand{\bbN}{\mathbb{N}}

\newcommand{\bbZ}{\mathbb{Z}}
\newcommand{\bbR}{\mathbb{R}}
\newcommand{\bbC}{\mathbb{C}}
\newcommand{\bbP}{\mathbb{P}}

\newcommand{\bbQ}{\mathbb{Q}}

\def\bary{\begin{array}} 
\def\eary{\end{array}} 
\def\ben{\begin{enumerate}} 
\def\een{\end{enumerate}}
\def\bit{\begin{itemize}} 
\def\eit{\end{itemize}}
\def\nn{\nonumber} 

\newcommand{\cO}{\mathcal{O}}

\newcommand{\cC}{\mathcal{C}}

\newcommand{\cI}{\mathcal{I}}

\newcommand{\cX}{\mathcal{X}}

\def\beq{\begin{equation}}                     %
\def\eeq{\end{equation}}                       %
\def\bea{\begin{eqnarray}}                     
\def\eea{\end{eqnarray}}
\def\bary{\begin{array}} 
\def\eary{\end{array}} 
\def\ben{\begin{enumerate}} 
\def\een{\end{enumerate}}
\def\bit{\begin{itemize}} 
\def\eit{\end{itemize}}
\def\nn{\nonumber} 
\def\de {\partial}

\def\IP{{\mathbb P}}

\def\a{\alpha}

\theoremstyle{plain}

\newtheorem{thm}{Theorem}[section]
\newtheorem{lem}[thm]{Lemma}

\newtheorem{prop}[thm]{Proposition}
\newtheorem{conj}[thm]{Conjecture}

\newtheorem*{conj*}{Conjecture}
\newtheorem{cor}[thm]{Corollary}
\newtheorem*{cor*}{Corollary}

\newtheorem{defn}{Definition}[section]

\theoremstyle{definition}

\newtheorem{rem}[thm]{Remark}
\newtheorem{rmk}[thm]{Remark}

\DeclareMathOperator{\Hom}{Hom}

\newcommand{\GITl}[1]{\backslash \!\! \backslash _{\kern-.2em #1 \kern0.1em}}
\newcommand{\GIT}[1]{/\!\!/_{\kern-.2em #1 \kern0.1em}}

\renewcommand{\l}{\left}
\renewcommand{\r}{\right}
\newcommand{\bra}{\left\langle}
\newcommand{\ket}{\right\rangle}

\newcommand{\ev}{\operatorname{ev}}

\def\bred{\begin{color}{red}}
\def\ered{\end{color}}

\def\bes{\begin{subequations}}
\def\ees{\end{subequations}}

\usepackage{tikz}
\usetikzlibrary{decorations.pathreplacing,angles,quotes}

\newcommand\NN{\mathbb N}

\newcommand\PP{\mathbb P}
\newcommand\Q{\mathbb Q}

\newcommand\vdim{\operatorname{vdim}}

\newtheorem{lemma-definition}[theorem]{Lemma-Definition}

\theoremstyle{definition}

\theoremstyle{remark}

\numberwithin{equation}{section}
\numberwithin{figure}{section}

\newcommand{\ZZ} {\mathbb{Z}}
\newcommand{\QQ} {\mathbb{Q}}
\newcommand{\RR} {\mathbb{R}}

\newcommand {\shO}  {\mathcal{O}}

\newcommand {\Mult} {\operatorname{Mult}}

\newcommand {\ol} {\overline}

\DeclareMathOperator {\hhh} {H}

\def\mydate{\ifcase\month \or January\or February\or March\or
April\or May\or June\or July\or August\or September\or October\or 
November\or December\fi \space\number\day,\space\number\year}

\newcommand{\pone}{\IP^1}

\DeclareMathOperator{\tot}{Tot}

\DeclareMathOperator{\mmm}{M}

\DeclareMathOperator{\obstr}{Ob}

\newcommand{\tpd}{\text{T}(\mathfrak{p})^X_{\mathsf{d}}}
\newcommand{\tqd}{\text{T}(\mathfrak{q})^X_{\mathsf{d}}}

\begin{document}

\title{
On the log-local principle for the toric boundary}

\author{Pierrick Bousseau}

\address{\tiny
Universit\'e Paris-Saclay, CNRS, Laboratoire de mathématiques d'Orsay, 91405, Orsay, France}


\email{pierrick.bousseau@u-psud.fr}


\author{Andrea Brini}
\address{\tiny 
University of Sheffield, School of Mathematics and Statistics, S11 9DW, Sheffield, United Kingdom\\
On leave from CNRS, DR 13, Montpellier, France}
\email{a.brini@sheffield.ac.uk}

\author{Michel van Garrel}

\address{\tiny University of Birmingham, School of Mathematics, B15 2TT, Birmingham, United Kingdom}
\email{m.vangarrel@bham.ac.uk}

\thanks{This project has been supported by the European Union's Horizon
  2020 research and innovation programme under the Marie Sklodowska-Curie
  grant agreement No 746554 (M.~vG.), the Engineering and Physical Sciences
  Research Council under grant agreement ref.~EP/S003657/2 (A.~B.) and by Dr.\ Max R\"ossler, the Walter Haefner Foundation and the ETH Z\"urich Foundation (P.~B. and M.~vG.).}

\begin{abstract}
Let $X$ be a smooth projective complex variety and let $D=D_1+\cdots+D_l$ be a
reduced
normal crossing divisor on $X$ with each
component $D_j$ smooth, irreducible, and nef. The log-local
principle put forward in \cite{vGGR} conjectures that the genus 0 log
Gromov--Witten theory of maximal tangency of $(X,D)$ is equivalent to the
genus 0 local Gromov--Witten theory of $X$ twisted by
$\bigoplus_{j=1}^l\shO(-D_j)$. We prove that an extension of the log-local
principle holds for $X$ a (not necessarily smooth) $\mathbb{Q}$-factorial projective toric variety, $D$ the toric boundary, and descendent point insertions.
\end{abstract}

\maketitle

\section{Introduction}

Let $X$ be a smooth projective complex variety of dimension $n$ and let
$D=D_1+\cdots+D_l$ be an effective reduced normal crossing divisor with each component $D_j$ smooth, irreducible and
nef. We can then consider two, a priori very different, geometries associated to the
pair $(X,D)$:
\begin{itemize}
\item the $n$-dimensional \emph{log geometry} of the pair $(X,D)$, 
\item the $(n+l)$-dimensional \emph{local geometry} of the total space $\tot\left(\bigoplus_{j=1}^l\shO_X(-D_j)\right)$.
\end{itemize}

The genus zero log Gromov--Witten  invariants of $(X,D)$ virtually count {rational curves}
\[
f : \pone \to X
\]
of a fixed degree $f_*[\pone]\in\hhh_2(X,\ZZ)$,
with {insertions}, such as {passing through a number of general points},
and with prescribed intersections with $D$.
Such an $f$ is said to be \emph{of maximal tangency} if $f(\pone)$ meets each $D_j$ in only one point of full tangency.
On the other hand, the local Gromov-Witten theory of
$\tot\left(\bigoplus_{j=1}^l\shO_X(-D_j)\right)$ is a way to study the
local contribution of $X$ to the enumerative geometry of a compact
$(n+l)$-dimensional variety  $Y$ containing $X$ with normal bundle $\bigoplus_{j=1}^l\shO_X(-D_j)$.

The existence of a relation between the log and the local theory of $(X,D)$
was introduced by the \emph{log-local principle} of \cite[Conjecture 1.4]{vGGR}:

\begin{conj}\label{conj:vGGR}
Let $\mathsf{d}$ be an effective curve class such that $\mathsf{d}\cdot D_j>0$ for all 
$1 \leq j \leq l$. After dividing by $\prod_{j=1}^l (-1)^{\mathsf{d}\cdot D_j +1}\, \mathsf{d}\cdot D_j$, the genus 0 log Gromov-Witten invariants of maximal tangency and class $\mathsf{d}$ of $(X,D)$ equal the genus 0 local Gromov-Witten invariants of class $\mathsf{d}$ of $\tot\left(\bigoplus_{j=1}^l\shO_X(-D_j)\right)$ (with the same insertions).

\end{conj}

\begin{thm}[\cite{vGGR}] \label{thm:vGGR}

The log-local principle holds if $X$ is a smooth projective variety and $D$ is smooth and nef.

\end{thm}

There are two natural directions to generalise the log-local
principle further. The first is to investigate extensions to correspondences with other invariants. At the level
of BPS invariants \cite{MR1844627,MR3228298,CGKT1,CGKT2,CGKT3} this is proven for the pair of $\mathbb{P}^2$ and smooth cubic in \cite{Bou19a,Bou19b} and in higher genus in \cite{Bousseau:2020ckw}. In \cite{Bousseau:2020fus, Bousseau:2020ryp}, we extend the correspondences to the non-toric and higher genus/refined setting and include open Gromov--Witten invariants, their underlying open BPS counts, as well as quiver Donaldson--Thomas invariants to the set of correspondences. Another direction is the relationship between local and orbifold invariants \cite{TY20,BNTY21}.

The second natural question is to what extent the log-local principle generalises to
the case when $X$ and $D$ are not smooth: log Gromov-Witten theory is indeed well-defined for any pair $(X,D)$ which is log smooth, 
but it is unclear how to define a local geometry in such generality. In the
present paper, we consider a situation that goes beyond the smoothness assumptions of \cref{conj:vGGR} and where both log and local sides can be defined: we take for $X$ a 
$\Q$-factorial projective toric variety and for $D$ the toric
boundary divisor of $X$. As $X$ is 
$\Q$-factorial, it makes sense to require that the components $D_j$ of $D$ are nef.
We show in \cref{classtoric} that requiring each $D_j$ to be nef forces $X$ to be a product of fake
weighted projective spaces. While such an $X$ is not
necessarily smooth, and  $D$ is typically not normal crossing, 
$(X,D)$ can naturally be viewed as a log smooth variety, and so log Gromov-Witten invariants of $(X,D)$ are well-defined.
On the other hand, $X$ can be naturally viewed as a smooth Deligne--Mumford stack, and the local geometry $\tot\left(\bigoplus_{j=1}^l\shO_X(-D_j)\right)$ makes sense in the category of orbifolds.
The local Gromov-Witten invariants can be defined using orbifold
Gromov--Witten theory \cite{MR2450211}, and it thus makes sense to ask if the
genus $0$ log invariants of maximal tangency of such a pair $(X,D)$ are
related in the sense of \cref{conj:vGGR} to the corresponding
local invariants. Our main result is the following
\cref{thm:intro}; we refer to \cref{thm:ponen,thm:main_log,thm:main_loc,thm:main} for precise statements.

\begin{thm}
Let $X$ be a $\Q$-factorial projective toric variety and let $D$ be the toric boundary divisor of $X$. Assume that all the components $D_j$ of $D$ are nef.
Then the genus $0$ log Gromov-Witten invariants of maximal tangency of $(X,D)$, and the genus $0$ local Gromov-Witten invariants of $(X,D)$, both with descendent
point insertions, can be computed in closed form for all degrees. As a corollary, the log-local principle holds for the resulting invariants.
\label{thm:intro}
\end{thm}

Except for the well-studied case when $X=\pone$, the log and local Gromov--Witten invariants of $(X,D)$
are non-zero only for one, two or, provided $X=(\pone)^n$, three point insertions. For $X=(\pone)^n$ we prove an equality of virtual fundamental classes and refer to well-known techniques to compute the invariants. For the other cases, our proof
proceeds by calculating both sides to obtain explicit closed formulas
for these invariants for all $(X,D)$ (\cref{thm:main_log,thm:main_loc}). To compute the log
invariants we use the tropical correspondence result \cite{ManRu} and an
algorithm of \cite{ManRu19} for the tropical multiplicity.
The log-local principle of
\cref{conj:vGGR} then predicts an explicit formula in all degrees for the local
invariants, which we verify using local mirror symmetry techniques and a
reconstruction result from small to big quantum cohomology.

\subsection*{Relation to \cite{NR_new} and \cite{Bousseau:2020fus,Bousseau:2020ryp}}

After this paper was finished, we received the manuscript
 \cite{NR_new} where the
log-local principle is considered for simple normal crossings
divisors. The respective strategies have different flavours in the proof and complementary
virtues in the outcome: \cite{NR_new} consider the log/local correspondence for $X$ smooth and $D_j$ a hyperplane
section, with a beautiful geometric argument reducing the simple normal
crossings case to the case of smooth pairs, and with no restrictions on $X$. The
combinatorial pathway we pursued in the toric setting allows on the other
hand to relax the hypotheses on the smoothness of $X$, the normal crossings
nature of $D$, and the very
ampleness of $D_j$, and it lends itself to a wider application to the
case when $D$ is not the toric boundary and the refinement to
include all-genus invariants. We consider this specifically in the follow-up
papers \cite{Bousseau:2020fus,Bousseau:2020ryp}, where we prove the log-local principle for 
log Calabi--Yau surfaces with the components of the anticanonical divisor
smooth and nef and suitably reformulate it to, and verify it for, the higher genus theory in these cases. In addition, we extend the correspondences to include open Gromov--Witten invariants, the various underlying BPS counts, and quiver Donaldson--Thomas invariants.

\begin{rem}
In its most recent version, \cite{NR_new} gives a counter-example in principle to Conjecture \ref{conj:vGGR}. It is proven that there is a choice of (unspecified) insertion leading to a counter-example. The geometry however is not log Calabi--Yau and the insertion is not formed of point insertions. It remains open whether the conjecture holds in the more restrictive setting of a log Calabi--Yau variety with only point insertions. The present paper as well as \cite{Bousseau:2020ryp,Bousseau:2020fus} provide evidence for it.
\end{rem}

\subsection*{Acknowledgements}

We thank Helge Ruddat for getting this project started by asking us to compute
some local Gromov--Witten invariants of $\mathbb{P}^4$ and for helpful discussions all along. We thank Travis Mandel for detailed explanations on \cite{ManRu,ManRu19}. We are grateful to
Miles Reid and Andrea Petracci for some discussions on the classification of toric varieties with
nef divisors. We thank Makoto Miura for a discussion on Hibi
varieties. Finally, it is a pleasure to thank Navid Nabijou and Dhruv Ranganathan for discussions on their
parallel work \cite{NR_new}.

\section{Setup}

\subsection{Notation}

Let $X$ be a $\mathbb{Q}$-factorial projective toric variety of dimension $n_X$ and let 
$D=D_1+\cdots+D_{l_D}$ be the toric boundary divisor of $X$.
In the foregoing discussion, we write
$r_X\coloneqq\mathrm{rank}~\mathrm{Pic}(X)$ for the rank of the Picard group
of $X$, so that $l_D=n_X+r_X$, and $\chi_X=\chi(X)\coloneqq\dim_\bbC H(X,\bbC)$ for the dimension of the cohomology of $X$. 
The variety $X$ has a natural presentation as a GIT quotient
$\bbC^{n_X+r_X}\GIT{t}{((\bbC^\star)^{r_X}\times G_X)}$ for $G_X$ a finite
abelian group; for every $1\leq j \leq l_D$, we write $D_j$ for the divisor
corresponding to the 
$(\bbC^\star)^{r_X}\times G_X$ reduction to $X$ of the $j^{\rm th}$ coordinate
hyperplane in $\bbC^{n_X+r_X}$. Note in particular that $\sum_{j=1}^{l_D} D_j=-K_X$.

We also fix a further piece of notation, which will turn out to be convenient when
dealing with the book-keeping of indices for products of fake weighted projective
spaces. Let $m\in \bbN_0$. If $\mathsf{v}=(v_1, \dots, v_m)\in\bbN^m$ is a lattice
point in the non-negative $m$-orthant, we write
$|\mathsf{v}|=\sum_{i=1}^m v_i$ for its 1-norm; in the following we will consistently use
serif fonts for orthant points and italic fonts for their Cartesian
coordinates. For $R$ a finitely generated commutative monoid with generators
$\a_1, \dots, \a_m$, $x=\a_1^{j_1}\dots \a_m^{j_m}\in R$ a reduced word in $\a_i$, and $\mathsf{v}\in \bbN^m$, we write
$x^\mathsf{v}$ for the product $\prod_i \a_i^{j_i v_i}\in R$. 
We introduce partial orders on the $m$-orthant by saying that $\mathsf{v} \prec
\mathsf{w}$ (resp.\ $\mathsf{v} \preceq
\mathsf{w}$) if $v_i < w_i$ (resp. $v_i \leq w_i$) for all $i=1,\dots,
m$. Also, we will write
$Q^X_{ij}\in \bbZ$, 
$i=1,\dots,
    r_X$, $j=1,\dots, n_X+r_X$, 
for the weight of the $i^{\rm th}$ factor of the $(\bbC^\star)^{r_X}$
torus action on the $j^{\rm th}$ affine factor of $\bbC^{n_X+r_X}$.

\label{sec:setup}

\begin{defn}
A {\rm nef toric pair} $(X,D)$ is a pair given by $X$ a $\Q$-factorial
complex projective toric variety with toric boundary divisor $D=D_1+\cdots+D_{l_D}$, such that all the components $D_j$ are nef.
\label{def:neftp}
\end{defn}

Nefness of all the components $D_j$ of the toric bundary divisor imposes
strong conditions on $X$, as the  Proposition \ref{classtoric} below shows.

\begin{defn}
Let $X$ be a $\Q$-factorial projective toric variety, and let $\bbC^{n_X+r_X}\GIT{t}{((\bbC^\star)^{r_X}\times G_X)}$ be its natural GIT description. We say that $X$ is a fake weighted projective space 
if $\bbC^{n_X+r_X}\GIT{t}{(\bbC^\star)^{r_X}}$ is a weighted projective space.
\end{defn}

\begin{prop}\label{classtoric}

Let $X$ be a $\QQ$-factorial projective variety such that every effective divisor on $X$ is nef. Then $X$ is a product of fake weighted projective spaces.
\end{prop}

\begin{proof}
By \cite[Proposition 5.3]{FujiSato}, $X$ admits a finite surjective toric morphism
$\prod \PP^{n_i}  \to  X$. Let $\Sigma \subset N\otimes \RR$ be the fan of $\prod \PP^{n_i}$ and $\Sigma' \subset N'\otimes \RR$ the fan of $X$. Then we have an injective morphism of lattices $N \to N'$ of finite index. Identifying $N$ with its image in $N'$, $\Sigma=\Sigma'$. It follows that $X$ is the quotient of $\prod \PP^{n_i}$ by $N'/N$. Hence $X$ is a product of fake weighted projective spaces.
\end{proof}

By \cref{classtoric}, there is $\mathsf{n}_X\in\bbN^{r_X}$ such that $n_X=|\mathsf{n}_X|$ and $X$ is a product of $r_X$, $n_i:=(\mathsf{n}_X)_i$-dimensional
fake weighted projective spaces,
$$X=\prod_{i=1}^{r_X} \bbP^{G_i}\left(\mathsf{w}_X^{(i)}\right),$$
with $\mathsf{w}_X^{(i)} =((\mathsf{w}_X)_1^{(i)},\dots,
(\mathsf{w}_X)_{n_i+1}^{(i)}) \in\bbN^{n_i+1}$, which we may assume not to have any common factors, and
\[
\bbP^{G_i}\left(\mathsf{w}_X^{(i)}\right) := \bbP\left(\mathsf{w}_X^{(i)}\right)\GIT{t}G_i,
\]
for $G_i$ a finite abelian group.
Notice that, for fixed $i$ and defining $\varepsilon_i:=\sum_{k=1}^{i-1} (n_k+1)$, we have
\beq
Q^X_{i,j+\varepsilon_i}= \l\{ \bary{cc}
(\mathsf{w}_X)_j^{(i)} &  1\leq j \leq n_i+1, \\
0 & \mathrm{else,}
\eary
\r.
\eeq
independent of the $G_i$.
Let $H_i:=\mathrm{pr}_i^* c_1(\cO_{\bbP^{G_i}(\mathsf{w}^{(i)})}(1))$ denote the
pull-back to $X$ of the (orbi-) hyperplane class of the $i^{\rm th}$
factor of $X$ and let $H:=H_1 \dots H_{r_X}$. 
These generate the classical cohomology ring,
\beq
\hhh^\bullet(X, \bbC) = \frac{\bbC[H_1, \dots, H_{r_X}]}{\bra \l\{H_i^{n_i+1}\r\}_{i=1}^{r_X} \ket},
\eeq
which is independent of the $G_i$,
and we can take a homogeneous linear basis for $H^\bullet(X,\bbC)$ in the form
$\{ H^\mathsf{l} \}_{\mathsf{l}_i\leq n_i}$. Notice, in particular, that  
\[
[\mathrm{pt}]=\prod_{i=1}^{r_X}\big{|}G_i\big{|} \: \, \prod_{i,j}
\left(\mathsf{w}_X\right)^{(i)}_j \: H^{\mathsf{n}_X}.
\label{eq:ptHnx}
\]
Indeed, if $G_i$ is trivial, this follows from applying
\cite[Theorem~1]{MR0339247} to each component in the product; 
and if $G_i$ is non-trivial, then the extra factor comes from the
component-wise identification
$H^\bullet\left(\bbP^{G_i}(\mathsf{w}^{(i)}),\bbC\right)=H^\bullet(\bbP(\mathsf{w}^{(i)}),\bbC)^{G_i}$. We will also write $\mathsf{d}=(d_1,\dots,d_{r_X})$ for the curve class
$d_1H_1+\cdots+d_rH_r$
and 
\beq \mathsf{d}^{\mathsf{n}_X} := \prod_{i=1}^{r_X} d_i^{n_i}\,.\eeq We order the toric divisors $D_j$ of $X$, $j=1,\dots, |\mathsf{n}_X|+r_X$, in such a way that
\[
Q^X_{ij} = \stackrel{i^{\rm th}}{(0,\dots,0,1,0,\dots,0)}\cdot D_j,
\]
where the 1 is in
the $i^{\rm th}$ position.
Finally, we define
\beq
e^X_j(\mathsf{d}):=\sum_i Q^X_{ij} d_i = \mathsf{d}\cdot D_j, \quad e^X(\mathsf{d}):=\sum_{j=1}^{|\mathsf{n}_X|+r_X} e^X_j(\mathsf{d}) = -\mathsf{d}\cdot K_X.
\eeq

\subsection{Log Gromov-Witten invariants}
\label{sec:log}

Let $(X,D)$ be a nef toric pair and let $\mathsf{d}$ be an effective curve
class on $X$.\footnote{Note that unlike in Conjecture \ref{conj:vGGR} we do
  not require that $\mathsf{d} \cdot D_j >0$ for all $1 \leq j \leq l_D$.} For the definition of log Gromov-Witten invariants, we endow\footnote{We refer to \cite{vGOR} for an introduction to log geometry.}
$X$ with the divisorial log structure coming from
$D$, and view $(X,D)$ as a log smooth variety.  The log structure is used to
impose tangency conditions along the components $D_j$ of $D$: in this paper we consider
genus 0 stable maps into $X$ of class $\mathsf{d}$ that meet each component $D_j$ in
one point of maximal tangency $\mathsf{d}\cdot D_j$.
The appropriate moduli space $\ol{\mmm}^{\log}_{0,m}(X,D,\mathsf{d})$ of
genus $0$
$m$-marked maximally
tangent stable log maps was constructed (in all generality) in
\cite{GS13,Chen14,AbramChen14}. In this description, we have $m$ marked points that have tangency 0 with the boundary (interior marked points), and $l_D$ marked points with maximal tangency with each $D_j$ respectively. In case $\mathsf{d}\cdot D_j=0$ for some $j$, this means that the corresponding maximal tangency marked point is an interior marked point. There is a virtual fundamental class
\[
[\ol{\mmm}^{\log}_{0,m}(X,D,\mathsf{d})]^{\rm vir} \in \hhh_{2 \mathfrak{vdim}^{(X,D,\mathsf{d})}_{\rm log}} ( \ol{\mmm}^{\log}_{0,m}(X,D,\mathsf{d}) ),
\]
where
\begin{eqnarray*}
\mathfrak{vdim}^{(X,D,\mathsf{d})}_{\rm log} &=& -\mathsf{d} \cdot K_X+\dim X -3 + m - \sum_{j=1}^{l_D} (\mathsf{d}\cdot D_j-1) \\
  &=& n_X + m + l_D - 3 = 2n_X + r_X + m - 3.
  \label{eq:vdimlog}
\end{eqnarray*}
Evaluating at the marked points $p_i$ yields the evaluation maps
\[
\ev_i \colon \ol{\mmm}^{\log}_{0,m}(X,D,\mathsf{d}) \longrightarrow X.
\]
For $L_i$ the $i^{\rm th}$ tautological line
bundle on $\ol{\mmm}^{\log}_{0,m}(X,D,\mathsf{d})$,
whose fiber at $[f\colon(C,p_1,\dots,p_m)\to X]$
is the cotangent line of $C$ at $p_i$,
there are tautological classes
$\psi_i\coloneqq c_1(L_i)$.
We are interested in the calculation of the genus 0 log
Gromov--Witten invariants of maximal tangency of $(X,D)$ with 1, 2 or 3 point
insertions and $\psi$-class insertions at one point, defined as follows:
\begin{eqnarray}
    \label{eq:rpd}
R\mathfrak{p}^X_{\mathsf{d}} &:=& \int_{[\ol{\mmm}^{\rm log}_{0,1}(X,D,\mathsf{d})]^{\rm vir}}
\ev_1^*([{\rm pt}]) \cup \psi_1^{n_X+r_X-2},
\\
  \label{eq:rqd}
R\mathfrak{q}^X_{\mathsf{d}}  &:=& \int_{[\ol{\mmm}^{\rm log}_{0,2}(X,D,\mathsf{d})]^{\rm vir}}
\ev_1^*([{\rm pt}]) \cup \ev_2^*([{\rm pt}]) \cup \psi_2^{r_X-1}.
\end{eqnarray}

The invariant $R\mathfrak{p}^X_{\mathsf{d}}$ (resp.\ $R\mathfrak{q}^X_{\mathsf{d}}$) is a
virtual count of rational curves in $X$ of degree $\mathsf{d}=(d_1,\dots,d_{r_X})$ that
meet each toric divisor $D_j$ in one point of maximal tangency $\mathsf{d}\cdot D_j=\sum_{i=1}^r d_iQ^X_{ij}=e_j^X(\mathsf{d})$ and that pass through one
point in the interior with $\psi^{n_X+r_X-2}$ condition (resp. two
points in the interior, one of which with a $\psi^{r_X-1}$ condition). 

\begin{rmk}
Having a point condition on $X$ cuts down the dimension of the moduli space by $n_X$. Thus \eqref{eq:rpd} and \eqref{eq:rqd} cover all possible invariants with descendent point insertions except for two families of cases. For the first, one distributes the descendent insertions along both points in \eqref{eq:rqd}. Adapting the log calculations of Section \ref{sec:proof_log} to that case is left as an exercise to the reader, see also Remark \ref{remark} for the local side. The second family of cases concerns the invariants of $(\pone)^n$ with any number of marked points if $n=1$ and up to 3 marked points if $n\geq 2$. We treat $(\pone)^n$ separately in \cref{thm:ponen}.
 
\end{rmk}

\subsection{Local Gromov-Witten invariants}
\label{sec:loc}

Let $(X,D)$ be a nef toric pair as in \cref{def:neftp} and write $X_D^{\rm loc} := \mathrm{Tot}( \bigoplus_i \cO_X(-D_i))$ for the
target space of the local theory.
By \cref{classtoric}, we can view $X$ and $X_D^{\rm loc}$ as the coarse moduli
schemes of smooth Deligne--Mumford stacks $\cX$ and $\cX_D^{\rm loc}$ over $\bbC$, where
\bea
\cX &:=& \bigtimes_{i=1}^{r_X} \l[ \big(\bbC^{(n_X)_i} \setminus \{0\}\big) /
  (\bbC^\star\times G_i)\r], \nn \\
\cX_D^{\rm loc} &:=& \bigtimes_{i=1}^{r_X} \l[ \Big( \big( \bbC^{(n_X)_i} \setminus
  \{0\}\big) \times \bbC^{(n_X)_i+1}\Big) /
  (\bbC^\star\times G_i)\r].
\eea

Even though $X_D^{\rm loc}$  is not proper and may be singular, the locution
``Gromov--Witten theory of $X_D^{\rm loc}$'' receives a
meaning in terms of the orbifold Gromov--Witten theory of $\cX$ twisted by
$\bigoplus_i \cO_\cX(-D_i)$ 
\cite{MR1950941,MR2450211} and restricted
over its non-stacky part, and we refer the reader in particular to
\cite{MR2450211} for the relevant background on the Gromov--Witten theory of
Deligne--Mumford stacks. Let $\ol{\mmm}_{0,m}(\cX,\mathsf{d})$ be the moduli stack of twisted genus $0$
$m$-marked stable maps $[f:\cC\to \cX]$ with $f_{*}([\cC])=\mathsf{d}\in
H_2(\cX, \bbQ)$, where $\cC$
is an $m$-pointed twisted curve\footnote{This means that the
  coarse moduli space of $\cC$ is a pre-stable curve in the ordinary sense, with
cyclic-quotient stackiness allowed at special points, and satisfying
kissing (balancing) conditions for the stacky structures at the nodes. See
\cite[Section~4]{MR2450211} for more details.} \cite{MR2450211}, and write $\ol{\mmm}_{0,m}(X,\mathsf{d})$ for the substack of twisted stable
maps such that the image of all evaluation maps is contained in the zero-age
component of the (rigidified, cyclotomic) inertia stack of $\cX$.
The stack $\ol{\mmm}_{0,m}(\cX,\mathsf{d})$ can be equipped with a virtual fundamental class \cite[Section
  4.5]{MR2450211}, which induces a virtual fundamental class of pure
homological degree over the stack $\ol{\mmm}_{0,m}(X,\mathsf{d})$ of stable maps to the coarse moduli space,
\[
\left[\ol{\mmm}_{0,m}(X,\mathsf{d})\right]^{\rm vir}\in\hhh_{2 \mathfrak{vdim}^{(X,D,\mathsf{d})}}(\ol{\mmm}_{0,m}(X,\mathsf{d}),\bbQ),
\]
where
\[\mathfrak{vdim}^{(X,D,\mathsf{d})} :=-K_X\cdot\mathsf{d} + \dim X + m -
3= \mathfrak{vdim}^{(X,D,\mathsf{d})}_{\rm log}+e_X(\mathsf{d})-l_D.\]
  
Let now $\mathsf{d}$ be such that $\mathsf{d} \cdot D_j>0$ for all $1 \leq j \leq
l_D$. Then  
 $\hhh^0(\cC,f^{*}\bigoplus_{j=1}^{l_D}\shO_{X}(-D_j))=0$
for every twisted stable map $[f:\cC\to \cX]$ with $f_{*}([\cC])=\mathsf{d}$, and so 
$
\obstr_D \coloneqq R^{1}\pi_{*}f^{*}
\left(\bigoplus_{j=1}^{l_D} \shO_{X}(-D_j)\right)
$
is a vector bundle on $\ol{\mmm}_{0,m}(\cX,\mathsf{d})$, 
which is of rank $\sum_{j=1}^{l_D}(\mathsf{d}\cdot D_j-1)$ and has fibre  $\hhh^1(C,f^{*}
\bigoplus_{j=1}^{l_D}\shO_{X}(-D_j))$ at a stable map $[f:\cC\to
  \cX]$. Restricting to the zero-age component defines the virtual fundamental class
\begin{eqnarray}
\label{eq:dtvfc}
[\ol{\mmm}_{0,m}(X_D^{\rm loc},\mathsf{d})]^{\rm vir}\coloneqq[\ol{\mmm}_{0,m}(X,\mathsf{d})]^{\rm vir}\cap
c_{\rm top}\left(\obstr_D\right) \in
\hhh_{2(\mathfrak{vdim}^{(X,D,\mathsf{d})}+l_D-e_X(\mathsf{d}))}(\ol{\mmm}_{0,m}(X,\mathsf{d}),\QQ),
\end{eqnarray}
and we have
\[\vdim
 \ol{\mmm}_{0,m}(X_D^{\rm loc},\mathsf{d})=
 \mathfrak{vdim}^{(X,D,\mathsf{d})}-e_X(\mathsf{d})+l_D=\mathfrak{vdim}^{(X,D,\mathsf{d})}_{\rm log}.\]

The restriction to the untwisted sector gives well-defined evaluation maps $
\ev_i \colon \ol{\mmm}_{0,m}(X,\mathsf{d}) \longrightarrow X$, 
and there are tautological classes
$\psi_i\coloneqq c_1(L_i)$, where the fibre of $L_i$  at a stable map
$[f:\cC\to\cX]$ is given by the cotangent line to the coarse moduli space of
$\cC$ at the $i^{\rm th}$ point. The (untwisted) local Gromov--Witten invariants
of $(X,D)$ are then caps of pull-backs of classes in $H^\bullet (X,
\bbC)$ via the evaluation maps against the virtual fundamental
class \eqref{eq:dtvfc}.
In particular, the local counterparts of \eqref{eq:rpd} and \eqref{eq:rqd} are  defined by
\begin{eqnarray}
  \label{eq:pd}
  \mathfrak{p}^X_{\mathsf{d}} &:=&
  \int_{[\ol{\mmm}_{0,1}(X_D^{\rm loc},\mathsf{d})]^{\rm vir}} \ev_1^*([\mathrm{pt}]) \cup \psi_1^{n_X+r_X-2},
  \\
  \mathfrak{q}^X_{\mathsf{d}} &:=&
  \int_{[\ol{\mmm}_{0,2}(X_D^{\rm loc},\mathsf{d})]^{\rm vir}} \ev_1^*([\mathrm{pt}]) \cup \ev_2^*([\mathrm{pt}]) \cup \psi_2^{r_X-1}.
  \label{eq:qd}
\end{eqnarray}

\section{Main results}

\label{sec:statements}

We first consider the case of $(\pone)^n$ and treat the general case thereafter.

\begin{thm}\label{thm:ponen}
Conjecture \ref{conj:vGGR} holds for $X=(\pone)^n$ with its toric boundary.
\end{thm}
\begin{proof}
For $X=\pone$, the log-local principle (at the level of the virtual fundamental classes) is a direct consequence of \cite{vGGR} since the toric divisors are disjoint. For $X=(\pone)^n$ with $n\geq2$, we apply the log product formula \cite{Her19,Ran19} on the log side and the product formula \cite{Beh99} on the local side to obtain an equality of virtual fundamental classes.
\end{proof}

Note that computational techniques to compute the invariants of $X=(\pone)^n$ (with arbitrary numbers of point insertions if $n=1$) are well-developed. For example, using tropical correspondence results one may show that the maximal tangency 3-pointed invariants of $(\pone)^n$ are $\prod_{j=1}^{2n} \mathsf{d}\cdot D_j=\prod_{i=1}^n d_i^2$.

The following \cref{thm:main_log,thm:main_loc}
compute the log and local Gromov-Witten invariants defined in
\cref{sec:log,sec:loc} in all degrees for a nef toric pair $(X,D)$.

\begin{thm}\label{thm:main_log}
Let $(X,D)$ be a nef toric pair and let $\mathsf{d}$ be an effective curve class on $X$. If there is $j$ such that $\mathsf{d} \cdot D_j =0$, then $R\mathfrak{p}^X_{\mathsf{d}} = R\mathfrak{q}^X_{\mathsf{d}} =0$. If $\mathsf{d} \cdot D_j >0$ for all
$1 \leq j \leq l_D$, then we have 
\bea
R\mathfrak{p}^X_{\mathsf{d}} &=& 1,  \\
R\mathfrak{q}^X_{\mathsf{d}}  &=&  \prod_{i=1}^{r_X} \big{|}G_i\big{|} \,  \l(\prod_{i,j} \big(\mathsf{w}_X\big)^{(i)}_j\r) \mathsf{d}^{\mathsf{n}_X}.
\eea
\end{thm}

We write $\overset{\circ}{\prod}_j e^X_j(\mathsf{d})$ to mean the product of $e^X_j(\mathsf{d})$ over $j\in\{1,\dots,|\mathsf{n}_X|+r_X \, \big{|} \, e^X_j(\mathsf{d})\neq0\}$.

\begin{thm}\label{thm:main_loc}
Let $(X,D)$ be a nef toric pair and let $\mathsf{d}$ be an effective curve class on $X$. Then
\bea
\mathfrak{p}^X_\mathsf{d} &=&
\frac{(-1)^{e^X(\mathsf{d})-n_X-r_X}}{\overset{\circ}{\prod}_j e^X_j(\mathsf{d})} \, , \\
\mathfrak{q}^X_\mathsf{d}
&=& \prod_{i=1}^{r_X} \big{|}G_i\big{|} \,
 \l(\prod_{i,j}
\big(\mathsf{w}_X\big)^{(i)}_j\r)
\mathsf{d}^{\mathsf{n}_X}\mathfrak{p}^X_\mathsf{d} \nn \\
&=& \prod_{i=1}^{r_X} \big{|}G_i\big{|} \,
 \l(\prod_{i,j} \big(\mathsf{w}_X\big)^{(i)}_j\r) \mathsf{d}^{\mathsf{n}_X} \, \frac{(-1)^{e^X(\mathsf{d})-n_X-r_X}}{\overset{\circ}{\prod}_j e^X_j(\mathsf{d})}.
\eea
\end{thm}

We deduce from these the log-local principle proved in the present paper.

\begin{thm}\label{thm:main}
The log-local principle holds for nef toric pairs $(X,D)$ with descendent point insertions 
and with no assumptions on $\mathsf{d}\cdot D_j$. 
That is,
for every effective curve class $\mathsf{d}$, the log and local invariants are equal up to the factor 
\[
\prod_{j=1}^{l_D} (-1)^{\mathsf{d}\cdot D_j +1}\, \mathsf{d}\cdot D_j =  (-1)^{e^X(\mathsf{d})-|\mathsf{n}_X|-r_X} \prod_{j=1}^{|\mathsf{n}_X|+r_X} e^X_j(\mathsf{d}) \, .
\]
\end{thm}

\cref{thm:main} is a direct corollary of the combination of \cref{thm:ponen,thm:main_log,thm:main_loc}.
We will prove \cref{thm:main_log} using a  tropical
correspondence principle, and 
\cref{thm:main_loc} using an equivariant mirror theorem. We review these
technical tools in \cref{sec:review}, and explain how to apply them to the
proofs of \cref{thm:main_log,thm:main_loc} in \cref{sec:proof_log,sec:proof_loc} respectively.

\section{Computational methods}
\label{sec:review}

\subsection{The log side: tropical curve counts}
\label{sec:troprev}

Let $X=\prod_{i=1}^{r_X} \bbP^{G_i}(\mathsf{w}^{(i)})$ as in \cref{sec:setup}
and let $\Sigma\subset N_\RR$ be the fan of $X=X_\Sigma$;
here $N\simeq\ZZ^{|\mathsf{n}_X|}$ and $N_\RR:=N\otimes_\ZZ\RR$. 
Define furthermore $N_\QQ:=N\otimes_\ZZ\QQ$ and let $M:=\Hom(N,\ZZ)$ 
be the dual of $N$. 
Denote by $[D_1],\dots,[D_{|\mathsf{n}_X|+r_X}]$ the rays of $\Sigma$ corresponding to the irreducible effective toric divisors of $X$. 
We use correspondence results with tropical curve counts as developed in
\cite{Mikh05,NiSi,ManRu} (see \cite{vGOR} for an introduction) and state them
in the generality needed for our purposes.

Denote by $\ol{\Gamma}$ the topological realisation of a finite connected graph and by $\Gamma$ the complement of a subset of 1-valent vertices. We require that $\Gamma$ has no univalent and no bivalent vertices. The set of its vertices, edges, non-compact edges and compact edges is denoted by $\Gamma^{[0]}$, $\Gamma^{[1]}$, $\Gamma^{[1]}_\infty$ and $\Gamma^{[1]}_c$ respectively. $\Gamma$ comes with a weight function $w : \Gamma^{[1]} \to \ZZ_{\geq 0}$. The non-compact edges come with markings. Weight 0, resp.\ positive weight, non-compact edges are \emph{interior}, resp.\ \emph{exterior}, markings. There will be 1 or 2 interior point markings, which we denote by $P_1$ and $P_2$, and $|\mathsf{n}_X|+r_X$ exterior markings corresponding to the toric divisors, which we denote by $[D_1],\dots,[D_{|\mathsf{n}_X|+r_X}]$ as well.

\begin{defn}
A \emph{genus 0 degree $\mathsf{d}$ maximally tangent parametrised marked tropical curve} in $X$ consists of $\Gamma$ as above and a 
continuous map $h : \Gamma \to N_\RR$ satisfying
\begin{enumerate}
\item For $E\in\Gamma^{[1]}$, $h|_E$ is constant if and only if $w(E)=0$. Otherwise, $h|_E$ is a proper embedding into an affine line with rational slope.
\item Let $V\in\Gamma^{[0]}$ with $h(V)\in N_\QQ$. For edges $E\ni V$, denote by $u_{(V,E)}$ the primitive integral vector at $h(V)$ into the direction $h(E)$ (and set $u_{(V,E)}=0$ if $w(E)=0$). The \emph{balancing condition} holds:
\[
\sum_{E\ni V} w(E)\,u_{(V,E)}=0.
\]
\item For each exterior marking $D_j$, $h|_{D_j}$ is parallel to the ray $[D_j]$ and $w(D_j)=\mathsf{d}\cdot D_j$.
\item The first Betti number $b_1(\Gamma)=0$.
\end{enumerate}

If $(\Gamma',h')$ is another such parametrised tropical curve, then an  \emph{isomorphism} between the two is given by a homeomorphism $\Phi:\Gamma\to\Gamma'$ respecting the discrete data and such that $h=h'\circ\Phi$. A \emph{genus 0 degree $\mathsf{d}$ maximally tangent marked tropical curve} then is an isomorphism class of such.

Moreover, we say that an interior marking $E$ satisfies a \emph{$\psi^k$-condition} if $h(E)$ is a $k+2$-valent vertex.
\end{defn}

Denote by $\tpd$ the (moduli) space of genus 0 degree $\mathsf{d}$ maximally tangent tropical curves in $X$ with the interior marking equipped with a $\psi^{|\mathsf{n}_X|+r_X-2}$-condition passing through a fixed general point in $\RR^{|\mathsf{n}_X|+r_X}$. Denote by $\tqd$ the moduli space of genus 0 degree $\mathsf{d}$ maximally tangent tropical curves in $X$ with the two interior markings $P_1$ and $P_2$ mapping to two fixed general points in $\RR^{|\mathsf{n}_X|+r_X}$ and such that $P_2$ has a $\psi^{r_X-1}$-condition. We will see in \cref{prop:tpd,prop:tqd} that each of $\tpd$ and $\tqd$ consist of one element. Since $\tpd$ and $\tqd$ are finite hence, their elements are \emph{rigid} \cite[Definition 2.5]{ManRu19}.

Counts of tropical curves are weighted with appropriate multiplicities. There are a number of ways of defining the multiplicity $\Mult(\Gamma)$ of $\Gamma$. The version we use was formulated (for $X$ smooth) in \cite[Theorem 1.2]{ManRu19}. We state it for our setting. Set $A:=\ZZ[N]\otimes_\ZZ\Lambda^\bullet M$. For $n\in N$ and $\alpha\in\Lambda^\bullet M$, write $z^n\alpha$ for $z^n\otimes\alpha$ and $\iota_n\alpha$ for the contraction of $\alpha$ by $n$. Recall that if $\alpha\in\Lambda^s M$, then $\iota_n\alpha\in\Lambda^{s-1} M$. For $k\geq1$, define $\ell_k:A^{\otimes k}\to A$ via
\[
\ell_k(z^{n_1}\alpha_1\otimes\dots\otimes z^{n_k}\alpha_k):=z^{n_1+\,\cdots\, +n_k} \iota_{n_1+\,\cdots\, +n_k} (\alpha_1\wedge\cdots\wedge\alpha_k).
\]
Let now $h:\Gamma\to N_\RR$ be in $\tpd$ or $\tqd$ and choose a vertex $V_\infty$ of $\Gamma$. Consider the flow on $\Gamma$ with sink vertex $V_\infty$. To each edge $E$ of $\Gamma$, we inductively associate an element $\zeta_E=z^{n_E}\alpha_E\in A$, well-defined up to sign:
\begin{itemize} 
\item For the exterior markings, set $\zeta_{D_j}=z^{w(D_j)\Delta(j)}$, where $\Delta(j)$ is the primitive generator of $[D_j]$.
\item For an interior marking $P$, set $\zeta_P$ to be one of the two generators of $\Lambda^{|\mathsf{n}_X|}M$.
\item If $E_1,\dots,E_k$ are the edges flowing into a vertex $V\neq V_\infty$ and $E_{\text{out}}$ is the edge flowing out, set $\zeta_{E_{\text{out}}}=\ell_k({\zeta_{E_1}\otimes\cdots\otimes\zeta_{E_k}})$.
\end{itemize}
By \cite[Theorem 1.2]{ManRu19}, $\zeta_\Gamma:=\prod_{E\ni V_\infty}
\zeta_E\in z^0\otimes\Lambda^{|\mathsf{n}_X|}M$ and $\Mult(\Gamma)$ is the
index of $\zeta_\Gamma$ in $\Lambda^{|\mathsf{n}_X|}M$. It then follows from
\cite[Theorem 1.1]{ManRu} that $R\mathfrak{p}^X_{\mathsf{d}}$ is the number of
$\Gamma$ in $\tpd$ counted with multiplicity $\Mult(\Gamma)$, and
$R\mathfrak{q}^X_{\mathsf{d}}$ is the weighted cardinality of $\{\Gamma \in \tqd\}$, each weighted by $\Mult(\Gamma)$.

\begin{rem} Note that a priori \cite[Theorem 1.1]{ManRu} is stated for smooth varieties; in the cases of interest to us, however, the curves never meet the deeper toric strata and the arguments of \cite{ManRu} carry through.
\end{rem}

\subsection{The local side: mirror symmetry for toric stacks}
\label{sec:cg}

The second technical result we will use for the calculation of local Gromov-Witten invariants is
\cref{thm:I=J} below.
Consider
a torus $T\simeq \bbC^\star$ acting on $X^{\rm
  loc}_D \coloneqq\mathrm{Tot}\left(\bigoplus_i \cO_{X}(-D_i)\right)$ transitively on the fibres
and covering the trivial action on the image of the zero section. We will
denote by 
$\lambda\coloneqq c_1(\cO_{\bbP^\infty}(1))$ the polynomial generator of the
$T$-equivariant cohomology of a point, $\hhh_T(\mathrm{pt})=\hhh(BT)\simeq
\bbC[\lambda]$.
The basis elements $H^{\mathsf{l}}$ of \cref{sec:setup} for the cohomology of
$X$ have canonical $T$-equivariant lifts, which by a slight abuse of notation we denote with the same
symbol, to cohomology classes in
$X^{\rm loc}_D$ forming a $\bbC(\lambda)$ basis of
$\hhh_T(X^{\rm loc}_D)$, where as usual
$\bbC(\lambda)$ is the field of fractions of $\hhh_T(\mathrm{pt})$.
The $T$-equivariant cohomology 
$\hhh_T(X^{\rm loc}_D)$ is furthermore endowed with a non-degenerate,
symmetric bilinear form given by the restriction of the $T$-equivariant Chen--Ruan \cite[Section 2.1]{MR3414388} pairing on the untwisted
component of the inertia stack of $\cX_D^{\rm loc}$,
\beq
\eta_{\mathsf{l} \mathsf{m}} := (H^{\mathsf{l}}, H^{\mathsf{m}})_{X_D^{\rm loc}} := \int_{X} \frac{H^{\mathsf{l}} \cup
H^{\mathsf{m}}}{\cup_i \mathrm{e}_T(\cO_X(-D_i))},
\label{eq:grampd}
\eeq
where $\mathrm{e}_T$ denotes the $T$-equivariant Euler class.

Let now $\tau \in \hhh_T(X_D^{\rm loc})$. The {\it equivariant big
  $J$-function} of $X_D^{\rm loc}$ is the formal power series  
\beq
J_{\rm big}^{X^{\rm loc}_D}(\tau, z):= z+\tau+\sum_{\mathsf{d}\in
  \mathrm{NE}(X)}\sum_{n\in \bbZ^+}\sum_{\mathsf{l},\mathsf{m} \preceq \mathsf{n}_X} \frac{1}{n!}\bra \tau, \dots, \tau, \frac{H^{\mathsf{l}}}{z-\psi}
\ket_{0,n+1,\mathsf{d}}^{X_D^{\rm loc}} H^{\mathsf{m}} \eta^{\mathsf{l} \mathsf{m}} ,
\eeq
where we employed the usual correlator notation for Gromov-Witten invariants,
\beq
\bra \tau_1 \psi_1^{k_1}, \dots, \tau_n \psi^{k_n}_n\ket_{0,n,\mathsf{d}}^{X^{\rm
    loc}}:=\int_{[\ol{\mmm}_{0,m}(X^{\rm loc}_D,\mathsf{d})]^{\rm vir}} \prod_i \ev^*_i(\tau_i) \psi_i^{k_i},
\eeq
and $\eta^{\mathsf{l} \mathsf{m}} := (\eta^{-1})_{\mathsf{l}
  \mathsf{m}}$. Restriction to $t=t_0 \mathbf{1}_{H(X)} + \sum_{i=1}^{r_X} t_i H_i$ and use of the Divisor
Axiom leads to the {\it equivariant small $J$-function} of $X^{\rm
  loc}_D$,
\beq
J_{\rm small}^{X_D^{\rm loc}}(t, z):= z  \re^{\sum t_i \phi_i/z}\l(1+\sum_{\mathsf{d}\in
  \mathrm{NE}(X)} \sum_{\mathsf{l},\mathsf{m} \preceq \mathsf{n}_X} \re^{\sum t_i d_i} \bra \frac{H^{\mathsf{l}}}{z(z-\psi_1)}
\ket_{0,1,\mathsf{d}}^{X_D^{\rm loc}}  H^{\mathsf{m}} \eta^{\mathsf{l} \mathsf{m}}\r).
\label{eq:Jsmall}
\eeq
The $n$-pointed genus zero Gromov--Witten invariants with one marked descendant
insertion (respectively, the $1$-pointed genus zero descendant invariants) of
$X_D^{\rm loc}$, and no twisted insertions, can thus be read off
from the formal Taylor series expansion of $J_{\rm big}$ (resp., $J_{\rm
  small}$) at $z=\infty$.

The following theorem provides an explicit hypergeometric presentation of
$J_{\rm small}^{X_D^{\rm loc}}(t, z)$. Let $\kappa_j:=c_1(\cO(-D_j))$ be the $T$-equivariant first Chern
class of $\cO(D_j)$
and $y_i\in \mathrm{Spec}\bbC[[t]]$, $i=1, \dots, r_X$ be variables in a
formal disk around the origin. Writing $(x)_n:=\Gamma(x+n)/\Gamma(x)$ for the
Pochhammer symbol of $(x,n)$ with $n\in \bbZ$, the {\it $T$-equivariant $I$-functions} of  $X$ and $X_D^{\rm
  loc}$ are defined as the $\hhh_T(X)$ and $\hhh_T(X^{\rm loc}_D)$ valued Laurent series
\bea
\label{eq:Ifun}
  I^{X}(y,z) &:=&  z \mathbf{1}_{H(X)} +\prod_i y_i^{H_i/z} \sum_{\mathsf{d}\in \mathrm{NE}(X)}
  \prod_i y_i^{d_i} z^{\mathsf{d}\cdot K_X} \frac{1}{\prod_j \l(\frac{\kappa_j}{z}
    +1\r)_{\mathsf{d}\cdot D_j}},  \\
I^{X_D^{\rm loc}}(y,z) &:= & z \mathbf{1}_{H(X)} + \prod_i y_i^{H_i/z}
\sum_{\mathsf{d}\in \mathrm{NE}(X)}
  \prod_i y_i^{d_i} z^{\mathsf{d}\cdot(K_X+D)-l_D} \frac{\prod_j \kappa_j
    \l(\frac{\kappa_j}{z}+1\r)_{\mathsf{d}\cdot D_j-1}}{\prod_j \l(\frac{\kappa_j}{z}
    +1\r)_{\mathsf{d}\cdot D_j}}
  \label{eq:Ifunloc}
\eea
and their {\it mirror maps} as their formal $\cO(z^0)$ coefficient,
\begin{align}
  \tilde{t}^i_X(y) := & \big[z^0 H_i \big] I^{X}(y,z), \nn \\
  \tilde{t}^i_{X_D^{\rm loc}}(y) := & \big[z^0 H_i \big] I^{X_D^{\rm
      loc}}(y,z).
  \label{eq:mm}
\end{align}

Note that $\cX$ and $\cX_D^{\rm loc}$ are smooth toric Deligne--Mumford
stacks with coarse moduli schemes $X$ and $X_D^{\rm loc}$ that are
projective over their affinisation, and at this level of generality a result
of \cite{MR3414388} can be applied to provide a
Givental-style equivariant mirror statement for them, as follows. In the language of \cite{MR3414388}, the $I$-functions \eqref{eq:Ifun} and \eqref{eq:Ifunloc} are the
stacky $I$-functions of \cite[Definition~28 and 29]{MR3414388} for $\cX$ and
$\cX_D^{\rm loc}$ respectively, restricted
to insertions in the zero-age sector of their inertia stack. The main result
of \cite{MR3414388} identifies the small
$J$-function of a semi-projective toric Deligne--Mumford stack to its stacky
$I$-function, up to a change-of-variables given by its $\cO(z^0)$ term as in \eqref{eq:mm}. In
particular, the following statement is a projection to the untwisted
sector of $\cX$ and $\cX_D^{\rm loc}$ of \cite[Theorem~31 and Corollary 32]{MR3414388}.

\begin{thm}[\cite{MR3414388}]
  We have
  \begin{align}
    J_{\rm small}^{X}(\tilde t_X(y), z) = & I^{X}(y,z), \nn \\
    J_{\rm small}^{X_D^{\rm loc}}\l(\tilde t_X(y)+\tilde t_{X_D^{\rm loc}}(y)-\log
    y,
    z\r) = & I^{X_D^{\rm loc}}(y,z).
    \end{align}
\label{thm:I=J}
\end{thm}

\section{The log side: proof of \cref{thm:main_log}}
\label{sec:proof_log}

Assume first that there is $j$ such that $\mathsf{d} \cdot D_j =0$.
Given that
\[
X=\prod_{i=1}^{r_X} \bbP^{G_i}\left(\mathsf{w}_X^{(i)}\right)
\]
is given its toric boundary, each $D_j$ is of the form (up to reordering of the factors) 
\[
D^j \times \prod_{i\neq k} \bbP^{G_i}\left(\mathsf{w}_X^{(i)}\right)
\]
for some $k$ and with $D^j$ a prime toric divisor in $\bbP^{G_k}\left(\mathsf{w}_X^{(k)}\right)$.
As $d=(d_i)_i$, $d\cdot D_j=0$ implies by ampleness of $D^j$ that $d_k\cdot D^j=0$ and thus $d_k=0$.
This means that each genus 0 degree $\mathsf{d}$ maximally tangent stable log map factors through
\[
\bbP^{G_k}\left(\mathsf{w}_X^{(k)}\right) \; \times \; \prod_{i\neq k} \bbP^{G_i}\left(\mathsf{w}_X^{(i)}\right),
\]
and is trivial on the first component. By the log product formula \cite{Her19,Ran19}, the invariant reduces to the corresponding invariant of $\prod_{i\neq k} \bbP^{G_i}\left(\mathsf{w}_X^{(i)}\right)$. This moduli problem however is in positive virtual dimension 
and thus
$R\mathfrak{p}^X_{\mathsf{d}} = R\mathfrak{q}^X_{\mathsf{d}} =0$. For the remainder of this section, we therefore assume that $\mathsf{d} \cdot D_j >0$ for all $1 \leq j \leq l_D$.

Moving to the general case with 1 or 2 point insertions, recall that $X=\prod_{i=1}^{r_X} \bbP^{G_i}(\mathsf{w}^{(i)})$ 
is given by the fan $\Sigma\subset N_\RR$
where $N\simeq\ZZ^{|\mathsf{n}_X|}$ and $N_\RR=N\otimes_\ZZ\RR$. 
Then $\Sigma$ is the product fan of the fans $\Sigma_i\subset (N_i)_\RR$ of $\bbP^{G_i}(\mathsf{w}^{(i)})$, where $N_i\simeq\ZZ^{\mathsf{n}_i}$. Writing $\varepsilon_i:=\sum_{k=1}^{i-1} (n_k+1)$, the rays of $\Sigma_i$ are $[D_{\varepsilon_i+1}],\dots,[D_{\varepsilon_{i}+n_i+1}]$ and have primitive generators $\Delta(\varepsilon_i+1),\dots,\Delta(\varepsilon_i+n_i+1)$, which satisfy
\[
\mathsf{w}_{\varepsilon_i+1}^{(i)}\Delta(\varepsilon_i+1)+\cdots+\mathsf{w}_{\varepsilon_i+n_i+1}^{(i)}\Delta(\varepsilon_i+n_i+1)=0.
\]
Write $L_i$ for the sublattice of $N_i$ generated by the $[\Delta(\varepsilon_i+j)]$ and write $B_i$ for the change of basis matrix from a $\ZZ$-basis of $N_i$ to a $\ZZ$-basis of $L_i$. Then
\[
|\det B_i|=\big{|}N_i/L_i\big{|}=\big{|}G_i\big{|}.
\]
Let $L$ be the sublattice of $N$ generated by the $L_i$ and let $B$ be the change of basis matrix from $N$ to $L$ given by the $B_i$. We have that $|\det B|=\prod_{i=1}^{r_X}|\det B^i|=\prod_{i=1}^{r_X}|G_i|.$

\begin{prop}\label{prop:tpd}
The set $\tpd$ has an unique element $\Gamma$ of multiplicity $1$.
\end{prop}

\begin{proof}
Each element $\Gamma$ of $\tpd$ has $|\mathsf{n}_X|+r_X$ exterior markings (=rays) parallel to the rays $[D_1],\dots,[D_{|\mathsf{n}_X|+r_X}]$ and one vertex (=unique interior marking) with valency $|n|+r_X$. Thus the only possibility is that $\Gamma$ is the translate of the rays of the fan of $X$. Write $\zeta$ for one of the two generators of $\Lambda^{|\mathsf{n}_X|}M$. Then $\Mult(\Gamma)$ is given by the index of
\[
\prod_{j=1}^{|\mathsf{n}_X|+r_X} z^{e^X_j(\mathsf{d})\Delta(j)}=\zeta \in \Lambda^{|\mathsf{n}_X|}M
\]
in $\Lambda^{|\mathsf{n}_X|}M$, which equals 1.
\end{proof}

It follows from \cref{prop:tpd} and the correspondence result of \cite{ManRu}
that \[ R\mathfrak{p}^X_{\mathsf{d}} = 1. \]

We calculate the multiplicity of the element of
$T(\mathfrak{q})_\mathsf{d}^X$ in three steps of increasing generality.

\begin{prop}\label{prop:dim2}
Assume that $X$ is the fake weighted projective plane $\bbP^G(\mathsf{w}_1,\mathsf{w}_2,\mathsf{w}_3)$, where we assumed that $\gcd(\mathsf{w}_1,\mathsf{w}_2,\mathsf{w}_3)=1$. Then $\tqd$ has an unique element of multiplicity $|G| \, \mathsf{w}_1\mathsf{w}_2\mathsf{w}_3d^2$.
\end{prop}

\begin{proof}
From $\mathsf{w}_1\Delta(1)+\mathsf{w}_2\Delta(2)+\mathsf{w}_3\Delta(3)=0$, it follows that $|\Delta(1)\wedge\Delta(2)|=\mathsf{w}_3|\det B|$. Choose the basis $\{\Delta(1),\Delta(2)\}$ of $N_\RR$. In this basis, choose $P_1$ to be $(1,0)$ and $P_2$ to be $(0,1)$. Then the unique genus 0 degree $d$ maximally tangent tropical curve passing through $P_1$ and $P_2$ consists of the rays $[D_1],[D_2],[D_3]$, meeting at $0=(0,0)$, and with weights $\mathsf{w}_jd$ on $[D_j]$.

Choose 0 to be the sink vertex and let $E_1$, resp. $E_2$, be the edge connecting 0 with $P_1$, resp. $P_2$. Choose moreover $\{e_1,e_2\}$ to be a $\ZZ$-basis of $M$ with dual basis $\{e^*_1,e_2^*\}$. Then
\begin{gather*}
\zeta_{E_1} = \ell_2(\zeta_{D_1}\otimes\zeta_{P_1})=\ell_2(z^{\mathsf{w}_1d\Delta(1)}\otimes (e_1^*\wedge e_2^*))=z^{\mathsf{w}_1d\Delta(1)} \iota_{\mathsf{w}_1d\Delta(1)}(e_1^*\wedge e_2^*) \\ 
 = z^{\mathsf{w}_1d\Delta(1)} \left( (\iota_{\mathsf{w}_1d\Delta(1)}e_1^*)\wedge e_2^* - e_1^*\wedge \iota_{\mathsf{w}_1d\Delta(1)} (e_2^*) \right) = z^{\mathsf{w}_1d\Delta(1)} \left(  e_1^*(\mathsf{w}_1d\Delta(1)) \, e_2^* - e_2^*(\mathsf{w}_1d\Delta(1)) \, e_1^* \right).
\end{gather*}
Similarly
\[
\zeta_{E_2} = z^{\mathsf{w}_2d\Delta(2)} \left(  e_1^*(\mathsf{w}_2d\Delta(2)) \, e_2^* - e_2^*(\mathsf{w}_2d\Delta(2)) \, e_1^* \right)
\]
and 
\begin{align*}
\zeta_\Gamma &= \zeta_{D_3} \zeta_{E_1} \zeta_{E_2} = - e_1^*(\mathsf{w}_1d\Delta(1)) \,  e_2^*(\mathsf{w}_2d\Delta(2)) \, e_2^*\wedge e_1^* - e_2^*(\mathsf{w}_1d\Delta(1)) \, e_1^*(\mathsf{w}_2d\Delta(2)) \, e_1^*\wedge e_2^* \\
 &= \mathsf{w}_1 \mathsf{w}_2 d^2 \left( e_1^*(\Delta(1)) e_2^*(\Delta(2))  - e_2^*(\Delta(1)) e_1^*(\Delta(2))  \right) \, e_1^*\wedge e_2^* \\
 &= \mathsf{w}_1 \mathsf{w}_2 d^2 \, |\Delta(1)\wedge \Delta(2)| \, e_1^*\wedge e_2^* = \mathsf{w}_1 \mathsf{w}_2 \mathsf{w}_3 d^2 \, |\det B| \, e_1^*\wedge e_2^*,
\end{align*}
which is indeed of index $|G|\,\mathsf{w}_1\mathsf{w}_2\mathsf{w}_3d^2$ in $\Lambda^2M$.
\end{proof}

\begin{prop}
\label{prop:tqdr1}
Assume that $r_X=1$, i.e.\ $X=\bbP^G(\mathsf{w}_1,\dots,\mathsf{w}_{n+1})$ and that $n\geq3$.
Then, for an appropriate choice of marked points $P_1$ and $P_2$, the set $\tqd$ has a unique element $\Gamma$ of multiplicity $|G| \, \prod_{j=1}^{n+1}\mathsf{w}_j \, d^n$.
\end{prop}

\begin{proof}
We choose as basis of $N_\RR$ the basis $\{\Delta(1),\dots,\Delta(n)\}$. We choose our second point (interior marking) $P_2$ to have coordinate $(a_1,\dots,a_n)$ for $a_i<0$ and general. We choose our first marked point $P_1$ to have coordinate $(b,0,\dots,0)$ for $b>0$ large enough so that restricted to the halfspace $\{(x_1,\dots,x_n) | x_1 > b\}$, any $h\in\tqd$ is affine linear with image $(b,0,\dots,0)+\RR_{>0} \, \Delta(1)$ and weight $e_1(\mathsf{d})$.

For $1<j\leq n$, write $\mathsf{w}_{1j}\Delta(1j) := - \mathsf{w}_1\Delta(1)-\mathsf{w}_j\Delta(j)$ with $\Delta(1j)$ primitive and $\mathsf{w}_{1j}\in\NN$. Consider the finite abelian group
\[
G^j:=\left(\langle\Delta(1),\Delta(2)\rangle_\RR\cap N\right)/\langle\Delta(1),\Delta(2),\Delta(1j)\rangle.
\]
Given $\Gamma\in\tqd$, projecting to the plane $\langle\Delta(1),\Delta(2)\rangle_\RR$  leads to a genus 0 maximally tangent tropical curve in $\bbP^{G^j}(\mathsf{w}_1,\mathsf{w}_j,\mathsf{w}_{1j})$ passing through 2 general points. By \cref{prop:dim2}, there is only one such curve (and it has multiplicity $|G^j|\mathsf{w}_1\mathsf{w}_j\mathsf{w}_{1j}d^2$). These curves lift to a unique maximally tangent curve $h : \Gamma \to N_\RR$.

Choose $P_2$ to be the sink vertex and consider the associated flow. Since the $a_i$ are chosen to be general, on the set $\{ (x_i) | x_i < a_i \}$, $h$ is affine linear with slope parallel to $\Delta(n+1)$. We reorder the $\Delta(j)$ such that following the flow from $P_2$, the rays that are added to $\Gamma$ are successively translates of $[D_n],[D_{n-1}],\dots,[D_2]$. Note that all vertices are 3-valent since $P_1$ and $P_2$ are in general position. Starting at $P_1$ and following the flow, we label the compact edges successively $E_1,\dots,E_n$. Choose a $\ZZ$-basis $e_1,\dots,e_n$ of $N$.
Then
\[
\zeta_{E_1}=\ell_2(z^{\mathsf{w}_1d\Delta(1)}\otimes e_1^*\wedge\cdots \wedge e_n^*)=z^{\mathsf{w}_1d\Delta(1)}\iota_{\mathsf{w}_1d\Delta(1)} (e_1^*\wedge\cdots\wedge e_n^*).
\]
At the next step,
\begin{align*}
\zeta_{E_2}&=z^{\mathsf{w}_1d\Delta(1)+\mathsf{w}_2d\Delta(2)} \iota_{\mathsf{w}_1d\Delta(1)+\mathsf{w}_2d\Delta(2)} \circ \iota_{\mathsf{w}_1d\Delta(1)} (e_1^*\wedge\cdots\wedge e_n^*) \\
&=z^{\mathsf{w}_1d\Delta(1)+\mathsf{w}_2d\Delta(2)} \iota_{\left(\mathsf{w}_1d\Delta(1)+\mathsf{w}_2d\Delta(2)\right) \wedge \mathsf{w}_1d\Delta(1)} (e_1^*\wedge\cdots\wedge e_n^*) \\
&=z^{\mathsf{w}_1d\Delta(1)+\mathsf{w}_2d\Delta(2)} \iota_{\left(\mathsf{w}_1\mathsf{w}_2d^2\Delta(1)\wedge\Delta(2)\right)} (e_1^*\wedge\cdots\wedge e_n^*).
\end{align*}
Iterating this process, we obtain that
\[
\zeta_{E_n} = z^{\mathsf{w}_1d\Delta(1)+\cdots+\mathsf{w}_nd\Delta(n)} \iota_{\left(\mathsf{w}_1\cdots\mathsf{w}_nd^n\Delta(1)\wedge\cdots\wedge\Delta(n)\right)} (e_1^*\wedge\cdots\wedge e_n^*).
\]
Since $\mathsf{w}_1\Delta(1)+\cdots+\mathsf{w}_{n+1}\Delta(n+1)=0$, $|\Delta(1)\wedge\cdots\wedge\Delta(n)|=\mathsf{w}_{n+1}|\det B|$ and hence
\begin{align*}
\zeta_\Gamma &=  \iota_{\left(\mathsf{w}_1\cdots\mathsf{w}_nd^n\Delta(1)\wedge\cdots\wedge\Delta(n)\right)} (e_1^*\wedge\cdots\wedge e_n^*) \, e_1^*\wedge\cdots\wedge e_n^* \\
&= \mathsf{w}_1\cdots\mathsf{w}_n d^n \, |\Delta(1)\wedge\cdots\wedge\Delta(n)| \, e_1^*\wedge\cdots\wedge e_n^*\\
&=\mathsf{w}_1\cdots\mathsf{w}_n \mathsf{w}_{n+1} d^n \, |\det B| \, e_1^*\wedge\cdots\wedge e_n^*.
\end{align*}
which is indeed of index $|G|\, \mathsf{w}_1\cdots\mathsf{w}_n \mathsf{w}_{n+1} d^n$ in $\Lambda^{n}M$.
\end{proof}

\begin{prop}\label{prop:tqd}
Let  $X=\prod_{i=1}^{r_X} \bbP^{G_i}(\mathsf{w}^{(i)})$ be the product of fake weighted projective spaces. Then the set $\tqd$ has a unique element $\Gamma$ of multiplicity $\prod_{i=1}^{r_X} \big{|}G_i\big{|} \, \l(\prod_{i,j} \big(\mathsf{w}_X\big)^{(i)}_j\r) \mathsf{d}^{\mathsf{n}_X}$.
\end{prop}

\begin{proof}
Label the last $r_X$ divisors $D_{|\mathsf{n}_X|+1},\dots,D_{|\mathsf{n}_X|+r_X}$ to be coming from distinct components of $X$. Then $[D_{1}],\dots,[D_{|\mathsf{n}_X|}]\in\Sigma^{[1]}$ form a $\RR$-basis of $N_\RR$. To calculate $R\mathfrak{q}^X_{\mathsf{d}}$, we choose the marking $P_2$ with the $\psi^{r_X-1}$ condition to be the origin 0. We choose the marking $P_1$ a general point that has positive coordinates with respect to the above basis. Then the $r_X+1$ incoming rays at $P_2$ are necessarily $D_{|\mathsf{n}_X|+1},\dots,D_{|\mathsf{n}_X|+r_X}$ with weights $e_j^X(\mathsf{d})$ and a primitive vector in direction $-e_{|\mathsf{n}_X|+1}(\mathsf{d})D_{|\mathsf{n}_X|+1}-\dots-e_{|\mathsf{n}_X|+r_X}(\mathsf{d})D_{|\mathsf{n}_X|+r_X}$ with appropriate weight.

There is only one way to make a maximally tangent tropical curve $\Gamma$ passing through $P_1$ out of it. To see this, for each $i$, consider the map of fans $\Sigma\to\Sigma_i$ corresponding to the projection to the $i^{\rm th}$ component. In $\Sigma_i$ the tropical curve becomes straight at 0 and hence we are looking at maximally tangent curves of degree $d_i$ passing through two general points. By \cref{prop:tqdr1}, there is only one such. Moreover, the curve in $N_\bbR$ is uniquely determined by these projections.

Choose $P_2$ to be the sink vertex. Then the multiplicity of $\Gamma$ is calculated as in \cref{prop:tqdr1} to be $\prod_{i=1}^{r_X} \big{|}G_i\big{|} \, \l(\prod_{i,j} \big(\mathsf{w}_X\big)^{(i)}_j\r) \mathsf{d}^{\mathsf{n}_X}$.
\end{proof}

The correspondence principle of \cite{ManRu} then entails that \[ R\mathfrak{q}^X_{\mathsf{d}}= \prod_{i=1}^{r_X} \big{|}G_i\big{|} \, \l(\prod_{i,j} \big(\mathsf{w}_X\big)^{(i)}_j\r) \mathsf{d}^{\mathsf{n}_X}, \]
concluding the calculations of the logarithmic invariants of \cref{thm:main_log}.

\section{The local side: proof of \cref{thm:main_loc}}

\label{sec:proof_loc}

\subsection{The Poincar\'e pairing}

As in \cref{sec:cg}, we consider the scalar $T\simeq\bbC^\star$ action on $X^{\rm loc}_D$ that covers the
trivial action on the base $X$, and denote $\lambda=c_1(\cO_{B\bbC^\star}(1))$
for the corresponding equivariant parameter. Notice that for any
$\mathsf{l}\preceq \mathsf{n}_X$, the Gram matrix $\eta_{\mathsf{lm}}$ for the
restriction to the untwisted sector of the $T$-equivariant Chen-Ruan pairing \eqref{eq:grampd}
   of $X^{\rm loc}_D$ satisfies
\beq
\eta_{\mathsf{l} \mathsf{n}_X} = \int_{[X]}\frac{H^{\mathsf{l+n}_X}}{\mathrm{e}_T(N_{X/X^{\rm loc}_D})} =
\int_{[X]}\frac{H^{\mathsf{l+n}_X}}{(\mathrm{e}_T(N_{X/X^{\rm loc}_D}))^{[0]}}=
\int_{[X]}\frac{H^{\mathsf{l+n}_X}}{\prod_{i=1}^{|\mathsf{n}_X|+r_X} \lambda} = \l\{ \bary{cc}  
\frac{\prod_i |G_i| \, \prod_{i,j} (\mathsf{w}_X)_j^{(i)} }{ \lambda^{|\mathsf{n}_X|+r_X}} & \mathsf{l}=0, \\
0 & \mathrm{else,}\\
\eary \r.
\label{eq:etaXloc}
\eeq
for degree reasons. Also, $\eta_{\mathsf{l}\mathsf{m}}=0$ if
$|\mathsf{l}|+|\mathsf{m}|>|\mathsf{n}_X|$ for the same reason: this means
that 
$\eta_{\mathsf{lm}}$ is upper anti-triangular, and
$\eta^{\mathsf{lm}}:=(\eta^{-1})_{\mathsf{lm}}$ is lower anti-triangular with
anti-diagonal elements $\eta^{\mathsf{l},
  \mathsf{n}_X-\mathsf{l}}=1/\eta_{\mathsf{l}, \mathsf{n}_X-\mathsf{l}}$.

\subsection{One pointed descendants}

In the following, let $y=y_1 \dots y_{r_X}$ and $Q=\re^{t_1+\dots +
  t_{r_X}}$. From \eqref{eq:Jsmall}, we have
\beq
J_{\rm small}^{X^{\rm loc}_D}(t,z):=z \prod_{i=1}^{r_X} \re^{t_i
  H_i/z}\Bigg[1+\sum_{\mathsf{d},a,\mathsf{l},\mathsf{m}}  Q^{\mathsf{d}}
  z^{-a-2} \bra H^\mathsf{l} \psi^a \ket_{0,1,\mathsf{d}}^{X^{\rm loc}_D}
  \eta^{\mathsf{lm}} H^\mathsf{m}   \Bigg]=: \sum_{\mathsf{m}} (J_{\rm small}^{X_D^{\rm
    loc}})^{[\mathsf{m}]} H^\mathsf{m}.
\eeq
Using \eqref{eq:etaXloc},
we get that the component of the small, twisted $J$-function along
the identity class is
\bea
 (J_{\rm sm}^{X^{\rm loc}_D})^{[0]} &:=& z
\Bigg[1+\sum_{\mathsf{d},a,\mathsf{l}}  Q^{\mathsf{d}}
  z^{-a-2} \bra H^\mathsf{l} \psi^a \ket_{0,1,\mathsf{d}}^{X^{\rm loc}_D}
  \eta^{\mathsf{l}0} \Bigg], \nn \\
&=& z\Bigg[1+\frac{\lambda^{|\mathsf{n}_X|+r_X}}{\prod_i |G_i| \, \prod_{i,j}
    (\mathsf{w}_X)_j^{(i)}}\sum_{\mathsf{d},a}  Q^{\mathsf{d}}
  z^{-a-2} \bra H^{\mathsf{n}_X}\psi^a \ket_{0,1,\mathsf{d}}^{X^{\rm loc}_D}\Bigg], \nn \\
&=& z\Bigg[1+\lambda^{|\mathsf{n}_X|+r_X} \sum_{\mathsf{d},a}  Q^{\mathsf{d}}
  z^{-a-2} \bra [\mathrm{pt}] \psi^a \ket_{0,1,\mathsf{d}}^{X^{\rm loc}_D} 
  \Bigg].
\label{eq:Jsm1TB}
\eea
Therefore our first set of invariants \eqref{eq:pd} can be
computed from \eqref{eq:Jsm1TB} as
\beq
\mathfrak{p}^X_\mathsf{d} := \bra \mathrm{[pt]} \psi^{|\mathsf{n}_X|+r_X-2} \ket_{0,1,\mathsf{d}}=
\frac{1}{\lambda^{|\mathsf{n}_X|+r_X}} \l[z^{-|\mathsf{n}_X|-r_X} \re^{t\cdot \mathsf{d}}\r]
(J_{\rm small}^{X^{\rm  loc}_D})^{[0]}.
\label{eq:pdJ}
\eeq
To compute the r.h.s.\ we use \cref{thm:I=J}. For quantities $a(j)$ depending on $e_j^X(\mathsf{d})$, the notation $\overset{\circ}{\prod}_ja(j)$ refers to the product of $a(j)$ over $j\in\{1,\dots,|\mathsf{n}_X|+r_X \, \big{|} \, e^X_j(\mathsf{d})\neq0\}$. From \eqref{eq:Ifun} and \eqref{eq:Ifunloc}, the
$I$-functions of $X$ and 
$X_D^{\rm loc}$ are
\bea
\label{eq:IX}
I^{X}(y,z)
&:=& z \sum_\mathsf{d}\prod_{i=1}^{r_X} y_i^{H_i/z+d_i} 
\overset{\circ}{\prod}_j
\frac {1}{ \prod_{m_j=1}^{e^X(\mathsf{d})} \l(m_jz +\sum_{i} Q^X_{ij}H_i\r)}=: \sum_{\mathsf{m}} (I^{X_{\rm
    loc}})^{[\mathsf{m}]} H^\mathsf{m}, \\
I^{X^{\rm loc}_D}(y,z)
&:=& z \sum_\mathsf{d}\prod_{i=1}^{r_X} y_i^{H_i/z+d_i} 
\overset{\circ}{\prod}_j
\frac{\prod_{m_j=0}^{e_j^X(\mathsf{d})-1} \l(\lambda-m_jz -\sum_{i} Q^X_{ij} H_i\r)}{ \prod_{m_j=1}^{e^X(\mathsf{d})} \l(m_jz +\sum_{i} Q^X_{ij}H_i\r)}=: \sum_{\mathsf{m}} (I^{X_{\rm
    loc}})^{[\mathsf{m}]} H^\mathsf{m}. \nn \\
\label{eq:IXloc}
\eea
\begin{lem}
  The mirror maps of $X$ and $X^{\rm loc}_D$ are trivial,
  \beq
  \tilde{t}^X_i(y)=\tilde{t}_i^{X^{\rm loc}_D}(y)=\log y_i.
  \eeq
  
\end{lem}

\begin{proof}

  This is a straightforward calculation from \eqref{eq:IX} and
  \eqref{eq:IXloc}. Keeping track of the powers of $z$ in the general summands
  entails that 
$I^{X}(y,z) = z+ \sum_i \log y_i  H_i + \cO(1/z) = I^{X_D^{\rm
      loc}}(y,z)$, from which the claim follows.
\end{proof}

By the previous Lemma and \eqref{eq:pdJ}, to compute $\mathfrak{p}_\mathsf{d}^X$ we just need to evaluate the
component of the $I$-function of $X^{\rm loc}_D$ along the identity, divide by
$\lambda^{|\mathsf{n}_X|+r_X}$, and isolate the coefficient of
$\cO(z^{-|\mathsf{n}_X|-r_X})$.
We have
\bea
(I^{X^{\rm  loc}_D})^{[0]} &=& z \sum_\mathsf{d} y^\mathsf{d}
\frac{
\overset{\circ}{\prod}_j
\prod_{m_j=0}^{e_j^X(\mathsf{d})-1}
  \l(\lambda-m_jz\r)}{
  \overset{\circ}{\prod}_j
  \prod_{m_j=1}^{e_j^X(\mathsf{d})}
  \l(m_jz \r)} \nn \\
&=& z \sum_\mathsf{d} y^\mathsf{d}
\frac{
\overset{\circ}{\prod}_j
\prod_{m_j=0}^{e_j^X(\mathsf{d})-1}
  \l(\lambda-m_jz\r)}{z^{e^X(\mathsf{d})} 
  \overset{\circ}{\prod}_j (e_j^X(\mathsf{d}))!}.
\label{eq:I0}
\eea 
The numerator in the general
summand of \eqref{eq:I0} is divisible by $\lambda^{|\mathsf{n}_X|+r_X}$ (corresponding to
setting all $m_j=0$ in the product):
\beq
\overset{\circ}{\prod}_j
\prod_{m_j=0}^{e_j^X(\mathsf{d})-1}
  \l(\lambda-m_jz\r)= \lambda^{|\mathsf{n}_X|+r_X} 
  \overset{\circ}{\prod}_j
  \prod_{m_j=1}^{e_j^X(\mathsf{d})-1}
  \l(\lambda-m_jz\r),
\eeq
hence dividing by $\lambda^{|\mathsf{n}_X|+r_X}$  we get
\bea
\overset{\circ}{\prod}_j
\prod_{m_j=1}^{e_j^X(\mathsf{d})-1}
\l(\lambda-m_jz\r)
&=& (-z)^{e^X(\mathsf{d})-|\mathsf{n}_X|-r_X}
\l(
\overset{\circ}{\prod}_j
(e_j^X(\mathsf{d})-1)! +\cO(1/z)\r).
\eea
In particular this implies that 
\beq
\mathfrak{p}^X_\mathsf{d} = \bra \mathrm{[pt]} \psi^{|\mathsf{n}_X|+r_X-2} \ket_{0,1,\mathsf{d}}=
\frac{1}{\lambda^{|\mathsf{n}_X|+r_X}} \l[z^{-|\mathsf{n}_X|-r_X} y^\mathsf{d}\r]  (I^{X_{\rm
    loc}})^{[0]} =
\frac{(-1)^{e^X(\mathsf{d})-|\mathsf{n}_X|-r_X}}{\overset{\circ}{\prod}
  e^X_j(\mathsf{d})},
\eeq
proving the first part of \cref{thm:main}.

\subsubsection{Two pointed descendents}

Let us now turn to the computation of $\mathfrak{q}^X_\mathsf{d}$. We start with the following observation:
from \eqref{eq:IXloc}, we have
\beq
I^{X^{\rm loc}_D}(y,z)
:=z + \sum_{\mathsf{l}\preceq \mathsf{n}_X} \frac{1}{z^{|\mathsf{l}|-1}}
\l[ \prod_{i=1}^{r_X} \frac{\log^{l_i}y_i~
  H_i^{l_i}}{l_i!} + \cO\l(\frac{1}{z}\r)\r].
\label{eq:Ivan}
\eeq
This follows immediately from the fact that
\beq
\frac{\prod_{j=1}^{|\mathsf{n}_X|+r_X} \prod_{m_j=0}^{e_j^X(\mathsf{d})-1} \l(\lambda-m_jz -\sum_{i}
  Q^X_{ij} H_i\r)}{\prod_{j=1}^{|\mathsf{n}_X|+r_X} \prod_{m_j=1}^{e_j^X(\mathsf{d})} \l(m_jz
  +\sum_{i} Q^X_{ij}H_i\r)}= 
\cO\l( z^{-\sum_j \theta(e_j^X(\mathsf{d}))}\r),
\eeq
where $\theta(x)=0$ (resp. $\theta(x)=1$) for $x=0$ (resp. $x>0$).

From this we deduce the following Lemma. For $t \in \hhh_T(X_D^{\rm loc})$, let $\hat\star_t$
denote the big quantum cohomology product,
  \beq
  H^\mathsf{l}\hat \star_{t} H^\mathsf{m} :=
    \sum_{\mathsf{d}\in \mathrm{NE}(X)}\sum_{n \in
      \bbN}\sum_{\mathsf{i},\mathsf{k}\preceq \mathsf{n}} \bra  H^\mathsf{l},  H^\mathsf{m},
  H^\mathsf{i}, t, \dots, t \ket_{0,3+n,\mathsf{d}}^{X^{\rm loc}_D} \eta^{\mathsf{i}\mathsf{k}} H^\mathsf{k}
  \eeq
and $\star_y$ its restriction to small
quantum cohomology at $t=\sum_i \log y_i H_i$,
  \beq
  H^\mathsf{l}\star_{y} H^\mathsf{m} := \sum_{\mathsf{k}}
  c_{\mathsf{l}\mathsf{m}}^{\mathsf{k}}(y) H^\mathsf{k} :=
  \sum_{\mathsf{d}\in \mathrm{NE}(X)}\sum_{\mathsf{i},\mathsf{k}} \bra  H^\mathsf{l},  H^\mathsf{m},
  H^\mathsf{i} \ket_{0,3,\mathsf{d}}^{X^{\rm loc}_D} y^\mathsf{d} \eta^{\mathsf{i}\mathsf{k}} H^\mathsf{k}.
  \eeq
Write $t=\sum_{\mathsf{l}\preceq \mathsf{n}_X} t_{\mathsf{l}}
H^{\mathsf{l}}$. In the following we denote $\nabla_{H^\mathsf{l}} :=
\de_{t^{\mathsf{l}}}$ and, for any function $f: \hhh_T(X^{\rm loc}_D) \to \bbC(\lambda)$,
$f\big|_{\rm sqc}$ indicates its restriction to small quantum cohomology, $t \to
\sum_{i=1}^{r_X} t_i H_i$.

  \begin{lem}
\label{lem:recons}
    For $\mathsf{l}\prec \mathsf{n}_X$ we have
  \beq
  (\star_y)_{i=1}^{r_X} H_i^{\star_y l_i} = \cup_{i=1}^{r_X} H_i^{\cup l_i}=:H^\mathsf{l}.
  \eeq
  Moreover, 
  \beq
  z \nabla_{H^\mathsf{l}} J^{X^{\rm loc}_D}_{\rm big}(t,
  z)\big|_{\rm sqc}
  = \prod_i \Big(z y_i \de_{y_i} \Big)^{l_i} I^{X^{\rm loc}_D}(y, z).
  \eeq
\end{lem}

\begin{rmk}
  This proposition is a variation of the well-known statement that for
  $\bbP^n$ the small quantum product is the same as the cup product
  for all degrees up to and excluding $n$. 
\end{rmk}

\begin{proof}
 Recall that the components of $J^{X^{\rm loc}_D}_{\rm big}(t, z)$ are a
  set of flat coordinates for the Dubrovin connection in big quantum cohomology,
  \beq
  z \nabla_{H^\mathsf{l}}\nabla_{H^\mathsf{m}} J^{X^{\rm loc}_D}_{\rm big}(t,
  z) = \nabla_{H^\mathsf{l}\hat\star_{t} H^\mathsf{m}} J^{X^{\rm loc}_D}_{\rm big}(t,
  z).
  \eeq
Write now
  %
  $(\cI)_{[k]} := [z^{-k}] \cI$ for any Laurent series $\cI \in \bbC((z))$ and 
  %
suppose $|\mathsf{l}|=|\mathsf{m}|=1$. We have that
  \beq
  \nabla_{H^\mathsf{l}}\nabla_{H^\mathsf{m}} \big(J^{X^{\rm loc}_D}_{\rm
    big}(t)\big)_{[s]}\bigg|_{\rm sqc} = \nabla_{H^\mathsf{l}\star_y
    H^\mathsf{m}} (J^{X^{\rm loc}_D}_{\rm big}(t))_{[s-1]} =   \sum_{\mathsf{k}}
  c_{\mathsf{l}\mathsf{m}}^{\mathsf{k}}(y) \nabla_{\mathsf{k}}  \big(J^{X^{\rm loc}_D}_{\rm
    big}(t)\big)_{[s-1]}\bigg|_{\rm sqc}.
  \eeq
  Now,
  \beq
  c_{\mathsf{l}\mathsf{m}}^{\mathsf{k}}(y)=(y_l \de_{y_l}) (y_m \de_{y_m}) (J^{X^{\rm loc}_D}_{\rm sm}(y))^{[\mathsf{k}]}_{[1]}=\delta_{\mathsf{l}+\mathsf{m}}^{\mathsf{k}}  =c_{\mathsf{l}\mathsf{m}}^{\mathsf{k}}(0)
  \eeq
  from \eqref{eq:Ivan} and the fact that $(J^{X^{\rm loc}_D}_{\rm
    big})^{[\mathsf{k}]}$ is the $[\mathsf{k}]$-component of the gradient of
  the genus-0 Gromov-Witten potential. Then,
  \bea
  (y_l \de_{y_l}) (y_m \de_{y_m}) I^{X^{\rm loc}_D}_{[s]}(y) &=& \nabla_{H^\mathsf{l}}\nabla_{H^\mathsf{m}} \big(J^{X^{\rm loc}_D}_{\rm
    big}(t)\big)_{[s]}\bigg|_{\rm sqc} = \sum_{\mathsf{k}}
  c_{\mathsf{l}\mathsf{m}}^{\mathsf{k}}(0) \nabla_{\mathsf{k}}  \big(J^{X^{\rm loc}_D}_{\rm
    big}(t)\big)_{[s-1]}\bigg|_{\rm sqc} \nn \\ &=& \nabla_{H^\mathsf{l+m}}  \big(J^{X^{\rm loc}_D}_{\rm
    big}(t)\big)_{[s-1]}\bigg|_{\rm sqc}.
  \label{eq:dIstep1}
  \eea
  Now, for $|\mathsf{m}|=1$ and by induction on $1\leq |\mathsf{l}|<
  |\mathsf{n}_X|$ we have, from \eqref{eq:Ivan}, that
  %
  \beq
  c_{\mathsf{l}\mathsf{m}}^{\mathsf{k}}(y)
  =  \nabla_{H^{\mathsf{l}}}\nabla_{H^{\mathsf{m}}}  \big(J^{X^{\rm loc}_D}_{\rm
      big}(t)\big)^{[\mathsf{k}]}_{[1]}\bigg|_{\rm sqc}
= \l( \prod_i y_{l_i}  \de_{y_{l_i}}\r) (y_{m} \de_{y_{m}}) \big(I^{X^{\rm loc}_D}(y)\big)_{|\mathsf{l}|}^{[\mathsf{k}]}=\delta^{\mathsf{k}}_{\mathsf{l+m}}=c_{\mathsf{l}\mathsf{m}}^{\mathsf{k}}(0),
  \eeq
 and for $s\geq |\mathsf{l}|$,
  \bea
  \l( \prod_i y_{l_i}  \de_{y_{l_i}}\r) (y_{m} \de_{y_{m}}) \big(I^{X^{\rm loc}_D}(y)\big)_{[s]}
      &=&
      \nabla_{H^\mathsf{l}}\nabla_{H^\mathsf{m}} \big(J^{X^{\rm loc}_D}_{\rm
    big}(t)\big)_{[s-|\mathsf{l}|]}\bigg|_{\rm sqc} \nn \\ &=& \sum_{\mathsf{k}}
  c_{\mathsf{l}\mathsf{m}}^{\mathsf{k}}(0) \nabla_{\mathsf{k}}  \big(J^{X^{\rm loc}_D}_{\rm
    big}(t)\big)_{[s-|\mathsf{l}|+1]}\bigg|_{\rm sqc} \nn \\ &=& \nabla_{H^{\mathsf{l}+\mathsf{m}}}  \big(J^{X^{\rm loc}_D}_{\rm
    big}(t)\big)_{[s-|\mathsf{l}|+1]}\bigg|_{\rm sqc}.
  \label{eq:dIstep2}
  \eea

\end{proof}

\begin{cor}
  We have
  \beq
  \mathfrak{q}^X_\mathsf{d}= \prod_i \big{|} G_i \big{|} \, \l(\prod_{i,j} \big(\mathsf{w}_X\big)^{(i)}_j\r) \mathsf{d}^{\mathsf{n}_X}\mathfrak{p}^X_\mathsf{d} .
  \eeq
\end{cor}

\begin{proof}
From the previous Lemma we have, in particular, that
  \bea
  \nabla_{H^{\mathsf{n}_X}}  \big(J^{X^{\rm loc}_D}_{\rm
    big}(t)\big)_{[s]}\bigg|_{\rm sqc} = \l( \prod_i \Big(y_i \de_{y_i} \Big)^{n_i}\r)
\big(I^{X^{\rm loc}_D}(y)\big)_{[s+|\mathsf{n}_X|-1]}.
  \label{eq:dIfin}
  \eea
From \eqref{eq:qd} and \eqref{eq:dIfin} we have that
  \bea
  \mathfrak{q}^X_\mathsf{d} &=& [y^\mathsf{d}] \eta_{\mathsf{n}_X0} \nabla_{H^{\mathsf{n}_X}}  \big(J^{X^{\rm loc}_D}_{\rm
    big}(t)\big)^{[0]}_{[r_X+1]}\bigg|_{\rm sqc} \nn \\ &=&  \frac{ \prod_i |G_i| \, \prod_{i,j}
  \big(\mathsf{w}_X\big)^{(i)}_j}{ \lambda^{|\mathsf{n}_x|+r_X}} [y^\mathsf{d}] \prod_i \Big(y_i \de_{y_i} \Big)^{n_i}
  \big(I^{X^{\rm loc}_D}(y)\big)^{[0]}_{[|\mathsf{n}_X |+r_X]} \nn \\
  &=& \prod_i |G_i| \, \prod_{i,j} \big(\mathsf{w}_X\big)^{(i)}_j \prod_i d_i^{n_i} \, \mathfrak{p}^X_\mathsf{d} \, ,
  \eea
  concluding the proof.
\end{proof}

\begin{rmk}\label{remark}
The statement of \cref{lem:recons} also immediately reconstructs
explicitly two-point
descendent invariants where the powers of $\psi$-classes are distributed among
the two marked points by standard structure results about $g=0$ Gromov-Witten theory (namely the symplecticity of
the $S$-matrix, which is a consequence of WDVV and the string equation: this is
\cite[Lemma~17]{lee2004frobenius}). Their agreement with the corresponding log invariants is an easy
exercise left to the reader.
\end{rmk}

\bibliography{miabiblio.bib}
\end{document}